\documentclass[11pt,a4paper]{article}

\usepackage{mathrsfs}
\usepackage{amsmath}
\usepackage{amsfonts}
\usepackage{amssymb}
\usepackage{graphicx}
\usepackage{enumerate, color}
\usepackage{makeidx,epsf,psfrag,epsfig,graphics}
\usepackage{subfigure}

\allowdisplaybreaks

\begin{document}

\newtheorem{theorem}{Theorem}[section]
\newtheorem{corollary}[theorem]{Corollary}
\newtheorem{definition}[theorem]{Definition}
\newtheorem{conjecture}[theorem]{Conjecture}
\newtheorem{question}[theorem]{Question}
\newtheorem{lemma}[theorem]{Lemma}
\newtheorem{property}[theorem]{Property}
\newtheorem{proposition}[theorem]{Proposition}
\newtheorem{quest}[theorem]{Question}
\newtheorem{example}[theorem]{Example}
\newenvironment{proof}{\noindent {\bf
Proof.}}{\rule{2.5mm}{2.5mm}\par\medskip}
\newcommand{\de}{\em}
\newcommand{\ec}{{\rm ecc}}
\newcommand{\aec}{{\rm aecc}}
\newcommand{\sk}[1]{{\color{blue}#1}}

\title{ On the average Steiner 3-eccentricity of trees\thanks{Supported by the National Natural Science Foundation of China (No. 11861019), Guizhou Talent Development Project in Science and Technology (No. KY[2018]046), Natural Science Foundation of Guizhou (Nos.[2019]1047, [2018]5774-006, [2018]5774-021, [2020]1Z001), Foundation of Guizhou University of Finance and Economics(No. 2019XJC04). Sandi Klav\v zar acknowledges the financial support from the Slovenian Research Agency (research core funding No.\ P1-0297 and projects J1-9109, J1-1693, N1-0095, N1-0108. Corresponding author: Sandi Klav\v{z}ar.}}

\author{Xingfu Li$^a$, Guihai Yu$^a$, Sandi Klav\v{z}ar$^{b,c,d}$\\ \\
{\small  $^a$ College of Big Data Statistics, Guizhou University of Finance and Economics}\\
 {\small Guiyang, Guizhou, 550025, China}\\
{\small  {\tt xingfulisdu@qq.com; yuguihai@mail.gufe.edu.cn}}\\
{\small  $^b$ Faculty of Mathematics and Physics, University of Ljubljana, Slovenia}\\
{\small $^{c}$ Institute of Mathematics, Physics and Mechanics, Ljubljana, Slovenia} \\
{\small $^{d}$ Faculty of Natural Sciences and Mathematics, University of Maribor, Slovenia}\\
{\small {\tt sandi.klavzar@fmf.uni-lj.si}}
}

\maketitle

\begin{abstract}
The \emph{Steiner $k$-eccentricity} of a vertex $v$ of a graph $G$ is the maximum Steiner distance over all $k$-subsets of $V(G)$ which contain $v$. In this paper Steiner $3$-eccentricity is studied on trees. Some general properties of the  Steiner $3$-eccentricity of trees are given.  A tree transformation which does not increase the average Steiner $3$-eccentricity is given. As its application, several lower and upper bounds for the average Steiner $3$-eccentricity of trees are derived.
\end{abstract}


\medskip\noindent
\textbf{AMS Math.\ Subj.\ Class.\ (2010)}: 05C12, 05C05

\section{Introduction}

Throughout this paper, all graphs considered are simple and connected. If $G = (V(G), E(G))$ is a graph, then its order and size will be denoted by $n(G)$ and $m(G)$, respectively.
If $S\subseteq V(G)$, then the {\em Steiner distance} $d_G(S)$ of $S$ is the minimum size among all connected subgraphs of $G$ containing $S$, that is,
$$d_G(S) = \min\{ m(T):\ T\ {\rm subtree\ of}\ G\ {\rm with}\ S\subseteq V(T)\}\,.$$
If $k\ge 2$ is an integer and $v\in V(G)$, then the \emph{Steiner $k$-eccentricity} $\ec_{k}(v,G)$ of $v$ in $G$ is the maximum Steiner distance over all $k$-subsets of $V(G)$ which contain $v$, that is,
$$\ec_{k}(v,G) = \max\{d_{G}(S): \ v\in S\subseteq V(G), |S|=k\}\,.$$
Note that $\ec_{2}(v,G)$ is the standard eccentricity of the vertex $v$, that is, the largest distance between $v$ and the other vertices of $G$.

Li, Mao, and Gutman~\cite{Li2016Wiener} proposed the {\em $k$-th Steiner Wiener index} $SW_{k}(G)$ of $G$ as
$$SW_{k}(G) =  \sum_{S\in \binom{V(G)}{k}} d_{G}(S)\,.$$
Note that  $SW_{2}(G) = W(G)$, the celebrated Wiener index of $G$. Motivated by the $k$-th Steiner Wiener index, we introduce the {\em average Steiner $k$-eccentricity} $\aec_{k}(G)$ of $G$ as the mean value of all vertices' Steiner $k$-eccentricities in $G$, that is,
$$\aec_{k}(G) = \frac{1}{n(G)} \sum_{v\in V(G)} \ec_{k}(v,G)\,.$$
In this notation, $\aec_{2}(G)$ is just the standard average eccentricity of $G$, cf.~\cite{casablanca-2019, dankelmann2004, dankelmann2014, dankelmann2019, pei2019}).

The Steiner tree problem on general graphs is NP-hard to solve~\cite{Garey1979,Hwang1992},  but it can be solved in polynomial time on trees~\cite{Beineke1996}. The Steiner distance has been extensively studied on special graph classes such as trees, joins, standard graph products, corona products, and others, see~\cite{Anand2012, Chartrand1989, Gologranc, mao-2018, Wang2017}. The average Steiner $k$-distance and its close companion the $k$-th Steiner Wiener index have been studied on trees, complete graphs, paths, cycles, complete bipartite graphs, and others, see~\cite{Dankelmann1996, GutmanInPress}. The average Steiner distance and the Steiner Wiener index were also extensively studied, see~\cite{Dankelmann1997, Li2016Wiener, Li2018Wiener, lu, Wang2018, zhang-2019}. Some work on the Steiner diameter is present in~\cite{mao-2018, Wang2017}. Other topological indices related to the Steiner distance have also been investigated: Steiner Gutman index in~\cite{mao-2018b}, Steiner degree distance in~\cite{Gutman2016}, Steiner hyper-Wiener index in~\cite{Tratnik2019}, multi-center Wiener index in~\cite{Gutman2015}, Steiner Harary index in~\cite{mao-2017}, and Steiner (revised) Szeged index in~\cite{ghorbani}. Y.~Mao wrote an extensive survey paper on the Steiner distance in graphs~\cite{Mao2017Survey}.

In this paper we focus on the average Steiner $3$-eccentricity of trees. In the rest of this section we list additional definitions  needed in this paper. Then, in Section~\ref{sec:preliminary-results}, we present several structural properties of the Steiner $k$-eccentricity of trees and discuss the complexity of computing the average Steiner $3$-eccentricity of trees. In Section~\ref{sec:trans}, the average Steiner $3$-eccentricity of trees is investigated under a special transformation. Section \ref{trans-inform} presents the existence and a extremal graph of the $\pi$-transformation.  Relying on this behavior and the properties, in the subsequent section we establish several lower and upper bonds on the average Steiner $3$-eccentricity of trees. We conclude by presenting several topics for future research.

A vertex of a graph of degree $1$ is a \emph{leaf} or a \emph{pendent vertex}, and if it is of degree at least $2$, then it is an \emph{internal vertex}. With $\ell(G)$ we denote the number of leaves of a graph $G$. A vertex of a tree  of degree at least $3$ is a \emph{branching vertex}. An edge is \emph{pendent} if it is incident to a pendent vertex in a graph. A path $P$ of a graph $G$ is a \emph{pendent path} if  one endpoint of $P$ has degree $1$ and each internal vertex of $P$ has degree $2$.

If $H_1$ and $H_2$ are subgraphs of $G$, then the distance $d_G(H_1,H_2)$ between $H_1$ and $H_2$ is defined as $\min\{d_G(h_1, h_2):\ h_1\in V(H_1), h_2\in V(H_2)\}$. In particular, if $H_1$ is the one vertex graph with $u$ being its unique vertex, then we will write $d_G(u,H_2)$ for $d_G(H_1,H_2)$. The {\em eccentricity of a subgraph} $H$ in $G$ is $\ec_G(H) = \max\{d_{G}(v,H):\ v\in V(G)\}$.

If $S\subseteq V(G)$ and $T$ is subtree of $G$ with $S\subseteq V(T)$ and $m(T)  = d_G(S)$, then we say that $T$ is an {\em $S$-Steiner tree} and that a vertex of $S$ is a {\em terminal} of $T$. If $k\ge 2$ and $v\in V(G)$, then a $k$-set $S\subseteq V(G)$ is a \emph{Steiner $k$-ecc $v$-set} (or {\em $k$-ecc $v$-set} for short) if $v\in S$ and $d_{G}(S)=\ec_{k}(v,G)$; a corresponding tree that realizes $\ec_{k}(v,G)$ will be called a \emph{Steiner $k$-ecc $v$-tree} (or {\em $k$-ecc $v$-tree} for short). A vertex $v$ may have more than one $k$-ecc $v$-set, and each such set may have more than one Steiner $k$-ecc $v$-tree.

\section{Preliminary results}
\label{sec:preliminary-results}

The main topic of this paper is the average Steiner $3$-eccentricity (of trees). We first give exact values of it for some classes of graphs, easy computations being omitted.

\begin{proposition}
\label{aecc-on-special-graphs}
If $n\geq 3$, then $\aec_{3}(K_{n})=2$, $\aec_{3}(P_{n})=n-1$, $\aec_{3}(K_{1,n-1}) = 3-\frac{1}{n}$, and $\aec_{3}(C_{n})=\lceil\frac{3n}{4}\rceil$. Moreover, if $m, n\ge 3$, then $\aec_{3}(K_{m,n})=3$.
\end{proposition}

We now proceed with a series of lemmas.

\begin{lemma}\label{unique-k-ecc-tree}
If $T$ is a tree and $S\subseteq V(T)$, then the $S$-Steiner tree is unique.
\end{lemma}

Lemma~\ref{unique-k-ecc-tree} is implicitly used in the literature and also briefly mentioned in~\cite[p.~11]{Mao2017Survey}. It follows from the argument that two different $S$-Steiner trees would lead to a cycle in $T$. By Lemma~\ref{unique-k-ecc-tree}, the formulation of the next lemma is justified.

\begin{lemma}\label{distance-growing}
Let $T$ be a tree, $v\in V(T)$, and $v\in S\subseteq V(T)$, $|S| = k$. Let $T_{v}$ be the unique $S$-Steiner tree and $P$ a path in $T$, where $V(P)\not = \emptyset$, $V(P)\cap V(T_{v})=\{x\}$, and $x$ is an endpoint of the path $P$. If
\begin{itemize}
    \item[(1)] $x\in S$ and $x\ne v$, or
    \item[(2)] $x\notin S$ and $T_{v}$ has an internal vertex which is in $S$ and is different from $v$,
\end{itemize}
then there exists a $k$-set $S^{'}\not= S$ with $v\in S'$, such that the size of the $S^{'}$-Steiner tree is strictly larger than the size of $T_v$.
\end{lemma}

\begin{proof}
Suppose first that $x\in S$ and $x\ne v$. Let $u$ be the pendent vertex of $P$ not in $T_{v}$ and set $S^{'} = (S \cup\{u\})-x$. Then  the size of the $S'$-Steiner tree is $|E(T_{v})\cup E(P)|$. Since $|V(P)|\geq 2$, we have $|E(T_{v})\cup E(P)|\geq |E(T_{v})|+1>|E(T_{v})|$.

In the second case, let $t$ be the internal vertex of $T_v$ which is in $S$ and different from $v$. Let again $u$ be the pendent vertex of $P$ not in $T_{v}$. In this case we set $S^{'} = (S \cup\{u\})-t$ and obtain another $k$-set which induces a larger size Steiner tree than the original $k$-set $S$.
\end{proof}

Recall that $\ell(T)$ denotes the number of leaves of a tree $T$.

\begin{lemma}\label{all-leaves-must-be-terminals}
Let $T$ be a tree and $v\in V(T)$. If $k > \ell(T)$, then every $k$-ecc $v$-set contains all the leaves of $T$. The same conclusion holds if $v$ is a leaf and $k=\ell(T)$.
\end{lemma}

\begin{proof}
Trivially, $\ec_{k}(v,T) \leq n(T)-1$. Suppose that $k > \ell(T)$. Set $S = \{v\} \cup L\cup X$, where $L$ is the set of leaves of $T$ and $X$ a set of arbitrary $k-\ell(T) -1$ vertices from $V(T)\setminus (L\cup \{v\})$. Then $|S| = k$ and the  $S$-Steiner tree is the whole tree $T$.  Hence every $k$-ecc $v$-set is the whole tree $T$ and thus contains all the leaves. If $v$ is a leaf, then set  $S = L\cup X$, where $X$ a set of arbitrary $k-\ell(T)$ vertices from $V(T)\setminus L$ to reach the same conclusion.
\end{proof}

\begin{lemma}\label{terminals-must-be-leaves}
Let $T$ be a tree, $v\in V(T)$, and $\ell(T) \ge k\ge 2$. If $S$ is a $k$-ecc $v$-set, then every vertex from $S\setminus \{v\}$ is a leaf of $T$.
\end{lemma}

\begin{proof}
If $k = \ell(T)$ and $v$ is a leaf of $T$, then the conclusion follows by Lemma \ref{all-leaves-must-be-terminals}. In the rest we may hence assume that $k < \ell(T)$ or $v$ is not a leaf of $T$.

Let $T_{v}$ be a $k$-ecc $v$-tree and suppose that there exists a vertex $u\in S\setminus v$ which is an internal vertex of $T$. There there exists a leaf $x$ in $T$ which does not lie in $T_v$. Let $P$ be the unique $x,T_v$-path in $T$. Then $P$ is a pendent path with at least one edge not in $T_{v}$ and hence we can use Lemma~\ref{distance-growing} to obtain a larger $S$-Steiner tree, a contradiction.
\end{proof}

\begin{lemma}\label{longest-shortest-path-preserve}
Let $T$ be a tree and $v\in V(T)$. Then every Steiner $k$-ecc $v$-tree contains a longest path starting at $v$.
\end{lemma}

\begin{proof}
If $k=2$, then $\ec_{2}(v,T)$ is the length of a longest path from $v$ to all the other vertices in $T$, so there is nothing to be proved. In the sequel we may thus assume $k\geq 3$. Suppose on the contrary that $T_{v}$ is a $k$-ecc $v$-tree  which contains no longest path starting at $v$ in $T$. Let $S$ be the $k$-ecc $v$-set corresponding to $T_{v}$. Let $P$ be a longest path starting at $v$ in the tree $T$, and let $v^{''}$ be the endpoint of $P$ different from $v$. Let $P_{1}$ be the sub-path of $P$ which is shared by $T_{v}$, and $P_{2}$ be the remaining sub-path of $P$. Then $P_{1}$ and $P_{2}$ share a unique vertex $v^{'}\in V(P)$. The described situation is illustrated in Fig.~\ref{T-1}.

\begin{figure}[ht!]
\begin{center}
\epsfig{file=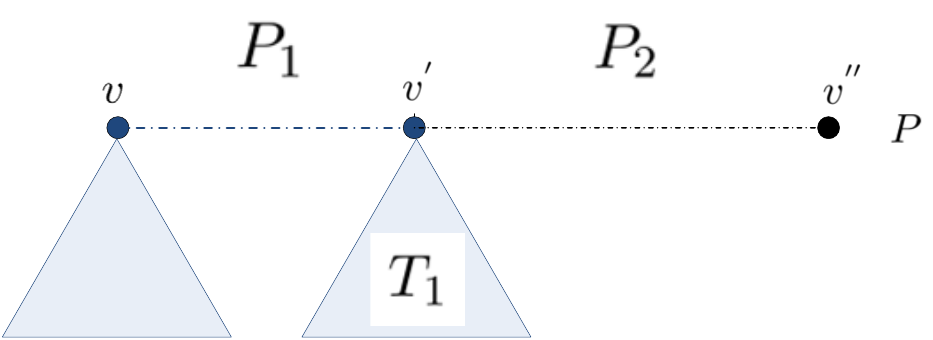, scale=0.5}
\end{center}
\vspace{-0.3cm}
\caption{The situation from the proof of Lemma~\ref{longest-shortest-path-preserve}; the grey part is $T_{v}$}
\label{T-1}
\end{figure}

Note that $v$ is an endpoint of $P_{1}$ and $v^{''}$ is an endpoint of $P_{2}$. By the assumption, $P_{2}$ is not empty. Let $F$ be a forest obtained by deleting all the edges in $E(P_{1})\subseteq E(T_{v})$ from the tree $T_{v}$. Let $T_{1}$ be the tree in $F$ which contains the vertex $v^{'}$, cf.\ Fig.~\ref{T-1} again. We now distinguish two cases.

Suppose first that $n(T_{1}) = 1$. Then $v^{'}$ is a leaf of $T_{v}$. So $v^{'}$ must be in the set $S$. We claim that $v^{'}\not =v$. Otherwise, the tree $T_{v}$ would be a trivial tree and $S$ contains the unique vertex $v$, which contradicts the fact that $k\geq 3$.   Let $S^{'} = (S\setminus\{v^{'}\}) \cup \{v^{''}\}$. Then $S^{'}$ is another $k$-set containing $v$ and its $S^{'}$-Steiner tree is $T_{v}\cup P_{2}$. Since $S^{'}$ is a larger tree than $T_{v}$, we have a contradiction to the fact that $T_{v}$ is a $k$-ecc $v$-tree.

Suppose second that $n(T_{1}) \ge 2$. Then there must be a vertex $u\in V(T_{1})$ such that $u$ is a leaf of $T_{v}$. Then $u$ lies in the $k$-set $S$. We construct a path $P_{3}$ as follows.
\begin{itemize}
\item If there is no branching vertex in $T_{v}$, then set $P_{3}$ to be the path from $v^{'}$ to $u$ in $T_{v}$.
\item Suppose that there is at least one branching vertex in $T_{v}$. Let $w\in V(T_{v})$ be the branching vertex nearest to $u$ in $T_{v}$. If $w$ is on the path from $v^{'}$ to $u$, then let $P_{3}$ be the path from $w$ to $u$ in the tree $T_{v}$. Otherwise, let $P_{3}$ be the path from $v^{'}$ to $u$ in the tree $T_{v}$.
\end{itemize}
Let $S^{'} = (S\setminus\{u\}) \cup \{v^{''}\}$. Then the tree $T_{v}^{'} = (T_{v}\setminus P_{3}) \cup P_{2}$ is the  $S^{'}$-Steiner tree. Since $P$ is a longest starting from $v$ and $T_v$ contains no such longest path from $v$,  the length of $P_{2}$ is strictly larger than the length of $P_{3}$. So $m(T_{v}^{'}) > m(T_{v})$, a final contradiction.
\end{proof}

In the rest of the section we focus on the structure of $3$-ecc $v$-trees. By Lemma~\ref{longest-shortest-path-preserve}, the endpoint $x$ of some longest path starting at $v$ must be in some $3$-ecc $v$-set. Here is now a property of the third terminal in a $3$-ecc $v$-set.

\begin{lemma}\label{third-terminal-property}
Let $v$ be a vertex of a tree $T$ and let $S = \{v,x,y\}$ be a $3$-ecc $v$-set, where the $v,x$-path $P$ is a longest path in $T$ starting from $v$. Then $d_{T}(y,P) = \ec_T(P)$.
\end{lemma}

\begin{proof}
Let $T_{v}$ be the $3$-ecc $v$-tree; so $T_{v}$ contains $P$ and the set $S = \{x,y,z\}$. The path $P$ is thus fixed and hence the vertex $y$ must be such that the $y,P$-path in $T$ is as long as possible. But this in turn implies that for the third terminal $y$ we must have $d_{T}(y,P) = \max\{d_{T}(s,P): s\in V(T)\} = \ec_T(P)$.
\end{proof}

Combining with Lemma~\ref{third-terminal-property}, the next lemma asserts that in the case of $3$-ecc $v$-sets, in Lemma~\ref{longest-shortest-path-preserve} an arbitrary longest path starting from $v$ can be used.

\begin{lemma}\label{equal-ecc-for-diff-longest-path-of-same-verte}
Let $v$ be a vertex of a tree $T$, and let $P_{1}$ and $P_{2}$ be distinct longest paths having $v$ as an endpoint. Then $\ec_T(P_{1}) = \ec_T(P_{2})$.
\end{lemma}

\begin{proof}
Let $w$ be the last common vertex of $P_1$ and $P_2$. Clearly $w$ exists, it is possible that $w=v$.  Let $t_1$ and $t_2$ be the other endpoints of $P_1$ and $P_2$, respectively. As $T$ is a tree and $P_1\ne P_2$ we have $t_1\ne t_2$. Let $u$ and $s$ be vertices of $T$ such that $d_{T}(u,P_{1}) = \ec_T(P_{1})$ and $d_{T}(s,P_{2}) = \ec_T(P_{2})$. Let further $P_{u}$ be the shortest $u,P_1$-path, $P_{s}$ the shortest $s, P_{2}$-path, $u_{0}$ the endpoint of $P_{u}$ different from $u$, and $s_{0}$ the endpoint of $P_{s}$ different from $s$.

We claim that $u_{0}$ lies in the $v,w$-subpath of $P_1$ (or $P_2$ for that matter). Suppose on the contrary that $u_0$ is an internal vertex of the $w,t_1$-subpath of $P_1$. Since $d_{T}(u,P_{1}) = \ec_T(P_{1})$ it follows that the length of $P_u$ is at least the length of the $w,t_2$-subpath of $P_2$. Since the latter path is of the same length as the $w,t_1$-subpath of $P_1$, we get that the concatenation of $P_u$ with the $v,u_0$-subpath of $P_1$ is a path strictly longer than $P_1$, a contradiction.

We have thus proved that $u_{0}$ lies in the $v,w$-subpath of $P_1$. By a parallel argument we also get that
$s_{0}$ lies in the $v,w$-subpath of $P_2$ (or in the  $v,w$-subpath of $P_1$ for that matter). But his means that $d_{G}(u,P_{1}) = d_{G}(s,P_{2})$ and hence $\ec_T(P_{1}) = \ec_T(P_{2})$.
\end{proof}

To conclude the section, we briefly discuss the computation complexity of the average Steiner 3-eccentricity on trees. The problem is clearly polynomial. If $T$ is a tree and $v$ its vertex, then we can first compute $\ec_3(v,T)$ by determining the $S$-Steiner tree for each $3$-set $S$ containing $v$, detecting in this way one of the largest size. This brute force strategy yields an $O(n^{4})$  algorithm. Based on the results of this section, we have devised a faster, $O(n^{2})$ algorithm. We do not present it here because in the follow up~\cite{Ilic2020OptimalAF} of our paper, Aleksandar Ili\'{c} has made full use of results from this section and elaborately devised a linear-time algorithm to calculate the average Steiner $3$-eccentricity of a tree. In another follow up paper~\cite{Li2020TheS}, the results of this section were extended to the Steiner $k$-eccentricity, $k\ge 3$, and presented a linear-time algorithm to calculate the Steiner $k$-eccentricity of a vertex in a tree. It would be interesting to see however, whether there is a linear-time algorithm to calculate the average Steiner $k$-eccentricity of a tree similar to Ili\'c's algorithm from~\cite{Ilic2020OptimalAF}.

\section{A transformation on trees}
\label{sec:trans}

Let $T$ be a tree with the structure as schematically depicted in Fig.~\ref{alpha-trans}. Here the $w,v_0$-path $P$ is a pendant path for which we require that $0\leq m(P) < \ec_{2}(u,T_{0})$ holds. (In case $m(P) = 0$, we have $v_0 = w$.) Then set $T^{'}=T\setminus\{wx:\ x\in N_{T_{1}}(w)\}\cup\{uy:\ y\in N_{T_{1}}(w)\}$, see Fig.~\ref{alpha-trans} again. We say that $T'$ is obtained from $T$ by a {\em $\pi$-transformation} and write $T' = \pi(T)$. The reverse transformation will be called a {\em $\pi^{-1}$-transformation}, that is, given $T'$ as in Fig.~\ref{alpha-trans}, we set $T=T^{'}\setminus\{ux:x\in N_{T_{1}}(u)\}\cup\{wy:y\in N_{T_{1}}(u)\}$ and write $T = \pi^{-1}(T')$.

\begin{figure}[htbp]
\vspace{0.5cm}
\begin{center}
\epsfig{file=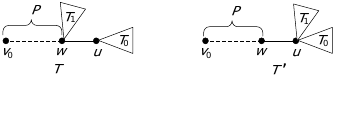, scale=2.0}
\end{center}
\vspace{-2.0cm}
\caption{$T$ and $T^{'}$}
\label{alpha-trans}
\end{figure}

\begin{theorem}\label{alpha-trans-not-decrease}
Let $T$ be a tree as in Fig.~\ref{alpha-trans}, and let $T' = \pi(T)$. Let $P_{0}$ be a longest path starting at $u$ in $T_{0}$. If $\ec_{T_0}(P_0) \leq \ec_{2}(w,P) < \ec_{2}(u,T_{0})$ and $\ec_{2}(w,T_{1}) \leq \ec_{2}(w,P)$, then $\aec_{3}(T^{'}) = \aec_{3}(T)$. Otherwise, $\aec_{3}(T^{'}) < \aec_{3}(T)$.
\end{theorem}

\begin{proof}
We are going to consider the behavior of the Steiner $3$-eccentricity on the sets of vertices $V(P)\setminus \{w\}$, $V(T_{1})$, and $V(T_{0})$ on the following cases. (Recall that in the definition of the $\pi$-transformation we have required that  $m(P)  = \ec_{2}(w,P) < \ec_{2}(u,T_{0})$ holds.)

In order to cover all possibilities, the proof is based on the relationship among $\ec_{2}(w,T_{1})$, $\ec_{2}(w,P)$ and $\ec_{2}(u,T_{0})$. By the definition, we have $\ec_{2}(w,P) < \ec_{2}(u,T_{0})$. So on the whole, there are three kinds of relationship among them, i.e., (1) $\ec_{2}(w,P)< \ec_{2}(u,T_{0})<\ec_{2}(w,T_{1})$, (2) $\ec_{2}(w,P)< \ec_{2}(w,T_{1})\leq \ec_{2}(u,T_{0})$, and (3) $\ec_{2}(w,T_{1})\leq \ec_{2}(w,P)<\ec_{2}(u,T_{0})$. In the following, we elaborate that the theorem holds for each of these three cases.

\medskip\noindent
{\bf Case 1}: $\ec_{2}(w,P)< \ec_{2}(u,T_{0})<\ec_{2}(w,T_{1})$.\\
To deal with this case, we need another variable $\ec_{T_1}(P_1)$, where $P_{1}$ is a longest path starting at $w$ in $T_{1}$. Since $\ec_{2}(w,T_{1})$ is the length of the longest path starting at $w$ in $T_{1}$, we have $\ec_{T_1}(P_1)\leq \ec_{2}(w,T_{1})$. Considering the relationship among $\ec_{2}(w,P)$, $ \ec_{2}(u,T_{0})$, $\ec_{2}(w,T_{1})$, and $\ec_{T_1}(P_1)$, there are the following three sub-cases to consider. 

\medskip\noindent
{\bf Case 1.1}: $0\leq \ec_{2}(w,P)< \ec_{2}(u,T_{0})<\ec_{T_1}(P_1)\leq \ec_{2}(w,T_{1})$.\\
In this case, the following three statements hold.
                \begin{enumerate}[(i)]
                    \item $\ec_{3}(v,T)-\ec_{3}(v,T^{'})=-1$ for every vertex $v\in V(P)\setminus \{w\}$.

                    By Lemmas~\ref{longest-shortest-path-preserve} and~\ref{third-terminal-property}, the other two terminals must be in $T_{1}$. This remains true after the  $\pi$-transformation is performed. So for every $v\in V(P)\setminus \{w\}$,  $\ec_{3}(v,T)$ increases by $1$ after the transformation.

                    \item $\ec_{3}(v,T)-\ec_{3}(v,T^{'})\geq0$  for every vertex $v\in V(T_{1})$.
                    \item $\ec_{3}(v,T)-\ec_{3}(v,T^{'})=1$  for every vertex $v\in V(T_{0})$.

                    For every vertex $v\in V(T_{0})$, by Lemma \ref{longest-shortest-path-preserve}, the other endpoint of a longest path starting at $v$ must be in $T_{1}$. This holds true after the $\pi$-transformation is performed. By Lemma~\ref{third-terminal-property}, the third terminal could not be $v_{0}$.  So after the $\pi$-transformation,  $\ec_{3}(v,T)$ decreases  by $1$.
                \end{enumerate}

Since we have assumed that $\ec_{2}(w,P)<\ec_{2}(u,T_{0})$, we have $|V(P)\setminus \{w\}| < n(T_{0})$.  In summary, in this case, we have
                \begin{align*}
                   \aec_{3}(T)-\aec_{3}(T^{'}) & = \frac{1}{n}\Big\{\mathop{\Sigma}\limits_{v\in V(P)\setminus \{w\}}[\ec_{3}(v,T)-\ec_{3}(v,T^{'})]+ \\
                                                & \ \ \ \mathop{\Sigma}\limits_{v\in V(T_{1}) }[\ec_{3}(v,T)-\ec_{3}(v,T^{'})] +  \mathop{\Sigma}\limits_{v\in V(T_{0})}[\ec_{3}(v,T)-\ec_{3}(v,T^{'})]\Big\}  \\
                                                &\geq  \frac{1}{n}\Big[|V(T_{0})|-|V(P)\setminus \{w\}|\Big] \\
                                                &>0\,.
                \end{align*}
We conclude that $\aec_{3}(T^{'}) < \aec_{3}(T)$ holds in this case.

\medskip\noindent
{\bf Case 1.2}: $0\leq \ec_{2}(w,P)<\ec_{T_1}(P_1)\leq \ec_{2}(u,T_{0})<\ec_{2}(w,T_{1})$. \\
Now we have:
                \begin{enumerate}[(i)]
                    \item $\ec_{3}(v,T)=\ec_{3}(v,T^{'})$ for every vertex $v\in V(P)\setminus \{w\}$.
                    \item $\ec_{3}(v,T)-\ec_{3}(v,T^{'})\geq0$  for every vertex $v\in V(T_{1})$.
                    \item $\ec_{3}(v,T)-\ec_{3}(v,T^{'})\geq0$  for every vertex $v\in V(T_{0})\setminus \{u\}$.
                    \item $\ec_{3}(u,T)-\ec_{3}(v,T^{'})=1$  for the vertex $u$.
                \end{enumerate}
Yet again $\aec_{3}(T^{'}) < \aec_{3}(T)$.

\medskip\noindent
{\bf Case 1.3}: $0\leq \ec_{T_1}(P_1)\leq \ec_{2}(w,P)< \ec_{2}(u,T_{0})<\ec_{2}(w,T_{1})$. \\
Now we have:
                \begin{enumerate}[(i)]
                    \item $\ec_{3}(v,T)=\ec_{3}(v,T^{'})$ for every vertex $v\in V(P)\setminus \{w\}$.
                    \item $\ec_{3}(v,T)-\ec_{3}(v,T^{'})\geq0$  for every vertex $v\in V(T_{1})$.
                    \item $\ec_{3}(v,T)-\ec_{3}(v,T^{'})\geq0$  for every vertex $v\in V(T_{0})\setminus \{u\}$.
                    \item $\ec_{3}(u,T)-\ec_{3}(v,T^{'})=1$.
                \end{enumerate}
Once more $\aec_{3}(T^{'}) < \aec_{3}(T)$.

So in Case 1 the theorem holds. In order to elaborate the other two cases, we use the variable $\ec_{T_0}(P_0)$ instead of $\ec_{T_1}(P_1)$ in the following proof. Since $P_{0}$ is a longest path starting at $u$ in $T_{0}$, we have $\ec_{T_0}(P_0)\leq \ec_{2}(u,T_{0})$.

\medskip\noindent
{\bf Case 2}: $\ec_{2}(w,P)< \ec_{2}(w,T_{1})\leq \ec_{2}(u,T_{0})$.\\
According to the relationship among $\ec_{2}(w,P)$, $\ec_{2}(w,T_{1})$, $\ec_{2}(u,T_{0})$ and $\ec_{T_0}(P_0)$, there are three sub-cases as following.

\medskip\noindent
{\bf Case 2.1}: $ 0\leq \ec_{2}(w,P)< \ec_{2}(w,T_{1})<\ec_{T_0}(P_0)\leq \ec_{2}(u,T_{0})$. \\
Now:
                \begin{enumerate}[(i)]
                    \item $\ec_{3}(v,T)=\ec_{3}(v,T^{'})$ for every vertex $v\in V(P)\setminus \{w\}$.
                    \item $\ec_{3}(v,T)-\ec_{3}(v,T^{'})=1$ for every vertex $v\in V(T_{1})$.
                    \item $\ec_{3}(v,T)-\ec_{3}(v,T^{'})\geq0$ for every vertex $v\in V(T_{0})$.
                \end{enumerate}
Again $\aec_{3}(T^{'}) < \aec_{3}(T)$.

\medskip\noindent
{\bf Case 2.2}: $ 0\leq \ec_{2}(w,P)<\ec_{T_0}(P_0)\leq \ec_{2}(w,T_{1})\leq \ec_{2}(u,T_{0})$. \\
Now:
                \begin{enumerate}[(i)]
                    \item $\ec_{3}(v,T)=\ec_{3}(v,T^{'})$ for every vertex $v\in V(P)\setminus \{w\}$.
                    \item $\ec_{3}(v,T)-\ec_{3}(v,T^{'})=1$  for every vertex $v\in V(T_{1})$.
                    \item $\ec_{3}(v,T)-\ec_{3}(v,T^{'})\geq 0$  for every vertex $v\in V(T_{0})\setminus \{u\}$.
                    \item $\ec_{3}(u,T)-\ec_{3}(v,T^{'})=1$ for the vertex $u$.
                \end{enumerate}
So also in this case $\aec_{3}(T^{'}) < \aec_{3}(T)$.

\medskip\noindent
{\bf Case 2.3}: $0\leq \ec_{T_0}(P_0)\leq \ec_{2}(w,P)< \ec_{2}(w,T_{1})\leq \ec_{2}(u,T_{0})$. \\
Now:
                \begin{enumerate}[(i)]
                    \item $\ec_{3}(v,T)=\ec_{3}(v,T^{'})$ for every vertex $v\in V(P)\setminus \{w\}$.
                    \item $\ec_{3}(v,T)-\ec_{3}(v,T^{'})\geq 0$ for every vertex $v\in V(T_{1})$.
                    \item $\ec_{3}(v,T)-\ec_{3}(v,T^{'})\geq 0$  for every vertex $v\in V(T_{0})\setminus \{u\}$.
                    \item $\ec_{3}(u,T)-\ec_{3}(v,T^{'})=1$.
                \end{enumerate}
We conclude that $\aec_{3}(T^{'}) < \aec_{3}(T)$ in this case.

\medskip\noindent
{\bf Case 3}: $\ec_{2}(w,T_{1})\leq \ec_{2}(w,P)<\ec_{2}(u,T_{0})$.\\
This case is similar to Case 2, there are also there sub-cases which are elaborated in the following.

\medskip\noindent
{\bf Case 3.1}: $0\leq \ec_{2}(w,T_{1})\leq \ec_{2}(w,P)<\ec_{T_0}(P_0)\leq \ec_{2}(u,T_{0})$.\\
Now:
\begin{enumerate}[(i)]
\item $\ec_{3}(v,T)=\ec_{3}(v,T^{'})$ for every vertex $v\in V(P)\setminus \{w\}$.
\item $\ec_{3}(v,T)-\ec_{3}(v,T^{'})=1$ for every vertex $v\in V(T_{1})$.
\item $\ec_{3}(v,T)=\ec_{3}(v,T^{'})$ for every vertex $v\in V(T_{0})$.
\end{enumerate}
Therefore,  $\aec_{3}(T^{'}) < \aec_{3}(T)$.

\medskip\noindent
{\bf Case 3.2}:  $0\leq \ec_{2}(w,T_{1})<\ec_{T_0}(P_0)\leq \ec_{2}(w,P)<\ec_{2}(u,T_{0})$.\\
In this case we obtain the same conclusions as the Case 1. Hence we conclude that $\aec_{3}(T^{'}) = \aec_{3}(T)$ holds also in this case.

\medskip\noindent
{\bf Case 3.3}: $0\leq \ec_{T_0}(P_0)\leq \ec_{2}(w,T_{1})\leq \ec_{2}(w,P)<\ec_{2}(u,T_{0})$.\\
In this case it is evident that the following three statements hold.
\begin{enumerate}[(i)]
\item $\ec_{3}(v,T)=\ec_{3}(v,T^{'})$ for every vertex $v\in V(P)\setminus \{w\}$.
\item $\ec_{3}(v,T)=\ec_{3}(v,T^{'})$ for every vertex $v\in V(T_{1})$.
\item $\ec_{3}(v,T)=\ec_{3}(v,T^{'})$ for every vertex $v\in V(T_{0})$.
\end{enumerate}
By the definition of the average Steiner $3$-eccentricity it thus follows that $\aec_{3}(T^{'}) = \aec_{3}(T)$ holds in this case.

Since all the possibilities have been verified, the theorem holds.
\end{proof}

\begin{corollary}\label{inverse-alpha-trans}
Let $T'$ be a tree as in Fig.~\ref{alpha-trans}, and let $T = \pi^{-1}(T')$. Let $P_{0}$ be a longest path starting at $u$ in $T_{0}$. If $\ec_{T_0}(P_0) \leq \ec_{2}(w,P) < \ec_{2}(u,T_{0})$ and $\ec_{2}(w,T_{1}) \leq \ec_{2}(w,P)$, then $\aec_{3}(T^{'}) = \aec_{3}(T)$. Otherwise, $\aec_{3}(T) > \aec_{3}(T')$.
\end{corollary}

\section{More on the $\pi$-transformation}\label{trans-inform}

In the section, we present several properties of the transformation introduced in the previous section. In Section~\ref{pi-exist} we list conditions under which the transformation exists, while in Section~\ref{bi-star-sec} we study an extremal graph at which we would arrive after a sequence of $\pi$-transformations.

\subsection{The existence of $\pi$-transformation}\label{pi-exist}
If a vertex of a tree has degree at least $3$, then the vertex is said to be a \emph{branching vertex}. If two distinct pendent paths are attached to the same branching vertex, then the branching vertex is said to be a \emph{pendent branching vertex}. Referring to the definition of the $\pi$-transformation in Section~\ref{sec:trans}, if the subtree $T_{1}$ is a pendent path of $T$, then the $\pi$-transformation will be refereed to as a \emph{pendent $\pi$-transformation} on $T$.

\begin{lemma}\label{further-branch-ver}
If there are two distinct branching vertices $u$ and $v$ in a tree, then there are two distinct pendent branching vertices $u_{1}$ and $u_{2}$ such that the distance between $u_{1}$ and $u_{2}$ is not less than the distance between $u$ and $v$.
\end{lemma}

\begin{proof}
Let $P$ be the path between $u$ and $v$. Since $u$ and $v$ are distinct, there must be two distinct leaves $u'$ and $v'$ such that the path between $u$ and $u'$ shares no edge with $P$, and the path between $v$ and $v'$ shares no edge with $P$. Moreover, there is a pendent branching vertex on the path between $u$ and $u'$, while there is a pendent branching vertex on the path between $v$ and $v'$, because $u$ and $v$ are both branching vertices. Therefore, there are two distinct pendent branching vertices $u_{1}$ and $u_{2}$ such that the distance between $u_{1}$ and $u_{2}$ is not less than the distance between $u$ and $v$.
\end{proof}

\begin{lemma}\label{pendent-trans-1}
If there are  two distinct branching vertices of a tree $T$ such that the distance between them is at least $2$, then there exists a pendent $\pi$-transformation on $T$.
\end{lemma}

\begin{proof}
 By Lemma \ref{further-branch-ver}, there are two distinct pendent branching vertices $u_{1}, u_{2}\in V(T)$ such that the distance between $u_{1}$ and $u_{2}$ is at least $2$, cf.~Fig.~\ref{pendent-trans-fig}. 
 The paths between $v_{3}$ and $u_{1}$ and between $w_{1}$ and $u_{1}$ are two pendent paths attached to the pendent branching vertex $u_{1}$, while the paths between $v_{4}$ and $u_{2}$ and between $w_{2}$ and $u_{2}$ are two pendent paths attached to the pendent branching vertex $u_{2}$. Let $P$ and $Q$ be the pendent paths between $w_{1}$ and $u_{1}$ and between $w_{2}$ and $u_{2}$, respectively. The neighbors of $u_{1}$ and $u_{2}$ on the path between them are respectively marked as $v_{1}$ and $v_{2}$. Since the distance between $u_{1}$ and $u_{2}$ is more than $1$, we have $v_{1}\not = u_{2}$ and $v_{2}\not = u_{1}$.

 \begin{figure}[ht!]
\begin{center}
\epsfig{file=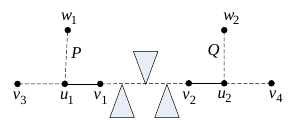, scale=1.5}
\end{center}
\vspace{-0.5cm}
\caption{The configuration of two pendent paths $P$ and $Q$}
\label{pendent-trans-fig}
\end{figure}

 If the length of the path between $v_{3}$ and $u_{1}$ is strictly less than the length of the path between $v_{1}$ and $v_{4}$, then a new tree $T'$ can be achieved by a pendent $\pi$-transformation, i.e., $T' = T\setminus\{(u_{1},x):x\in N_{P}(u_{1})\}\cup \{(v_{1},y):y\in N_{P}(u_{1})\}$.

 Otherwise, since $v_{1}\not = u_{2}$ and $v_{2}\not = u_{1}$, the length of the path between $v_{3}$ and $v_{2}$ is strictly larger than the length of the path between $u_{2}$ and $v_{4}$. Then a new tree $T''$ can be achieved by a pendent $\pi$-transformation, i.e., $T'' = T\setminus\{(u_{2},x):x\in N_{Q}(u_{2})\}\cup \{(v_{2},y):y\in N_{Q}(u_{2})\}$.
\end{proof}

As trees of course contain no triangles, Lemma~\ref{pendent-trans-1} immediately gives the following consequence. 

\begin{lemma}\label{pendent-trans-2}
If a tree $T$ has more than two branching vertices, then there exists a pendent $\pi$-transformation on $T$.
\end{lemma}

\begin{lemma}\label{pendent-trans-3}
Let $T$ be a tree with exact two branching vertices, where the two branching vertices are adjacent. If there are two pendent paths such that they have different length and attach to distinct branching vertices, then there exists a $\pi$-transformation on $T$. Otherwise, there is no $\pi$-transformation on $T$.
\end{lemma}
\begin{proof}
Let $u_{1}$ and $u_{2}$ be two unique (adjacent) branching vertices of $T$, see Fig.~\ref{two-branching-fig}.

Suppose first that there are two distinct leaves $v_{1}$ and $v_{2}$  such that the length of the path between $v_{1}$ and $u_{1}$ is not equal to length between $v_{2}$ and $u_{2}$. We may assume without loss of generality that the first of these paths is longer than the second. . Then the length of the path between $v_{2}$ and $u_{2}$ is strictly less than the length between $u_{2}$ and $v_{1}$.  Let $T_{1}$ be the subtree rooted at $u_{2}$ obtained by deleting all the edges on the path between $v_{1}$ and $v_{2}$, see Fig.~\ref{two-branching-fig} again. A new tree $T'$ can be obtained by a $\pi$-transformation on $T$ as $T' = T\setminus\{(u_{2},x):x\in N_{T_{1}}(u_{2})\}\cup\{(u_{1},y):y\in N_{T_{1}}(u_{2})\}$. 

 \begin{figure}[ht!]
\begin{center}
\epsfig{file=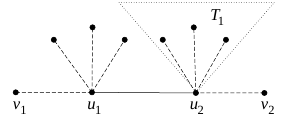, scale=1.5}
\end{center}
\vspace{-0.5cm}
\caption{The configuration of a tree with exact two branching verticces}
\label{two-branching-fig}
\end{figure}

In the second case all pendent paths attaching to both $u_{1}$ and $u_{2}$ have the same length. Then the precondition of the $\pi$-transformation does not hold and hence there is no $\pi$-transformation on $T$.
\end{proof}

\begin{lemma}\label{pendent-trans-4}
Let $T$ be a tree with exactly one branching vertex. If there are two distinct pendent paths such that their lengths differ by at least $2$, then here exists a $\pi$-transformation on $T$.
\end{lemma}

\begin{proof}
Let $u$ be the unique branching vertex of $T$. Let $v_{1}$ and $v_{2}$ be two leaves such that the length of the path between $v_{1}$ and $u$ has at least two more edges than the path between $v_{2}$ and $u$, see Fig.~\ref{one-branching-fig}.

\begin{figure}[ht!]
\begin{center}
\epsfig{file=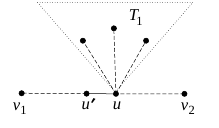, scale=1.5}
\end{center}
\vspace{-0.5cm}
\caption{The configuration of a tree with exact one branching verticces}
\label{one-branching-fig}
\end{figure}

Let $T_{1}$ be the subtree rooted at $u$ obtained by deleting all the edges on the path between $v_{1}$ and $v_{2}$. Let $u'$ be the unique neighbor of $u$ on the path between $u$ and $v_{1}$. Then the length of the path between $v_{1}$ and $u'$ is strictly larger than that between $u$ and $v_{2}$. So a new tree $T'$ can be obtained by a $\pi$-transformation on $T$ as $T' = T\setminus \{(u, x): x\in N_{T_{1}}(u)\}\cup\{(u',y):y\in N_{T_{1}}(u)\}$. 
\end{proof}

\subsection{Properties of a bi-stars}\label{bi-star-sec}

We say that a tree $T$ is a \emph{bi-star} if $T$ contains exactly two branching vertices, the two branching vertices are adjacent, and all its pendent paths are of the same length. (Note that in the literature bi-stars are often restricted to the case when all the pendent paths are of length $1$.) By Lemma \ref{pendent-trans-3},  bi-stars admit no $\pi$-transformation. In this subsection we study the Steiner $3$-eccentricity in a bi-stars. 

Let $u_{1}$ and $u_{2}$ be the two branching vertices in the bi-star $T$. Let $S_{1}$ and $S_{2}$ be the  subtrees respectively rooted at $u_{1}$ and $u_{2}$ obtained by deleting the edge $u_{1}u_2$. Then $S_{1}$ and $S_{2}$ are both star-like trees. By definition, all pendent paths in $S_{1}$ and $S_{2}$ have the same length. Let $v$ be a leaf in $V(S_{1})\cap V(T)$ and $w$ be a leaf in $V(S_{2})\cap V(T)$. Let $u_{1}'$ be the neighbor of $u_{1}$ on the path between $u_{1}$ and $v$ in $T$. Let $T_{1}$ be the sub-tree rooted at $u_{1}$ which is obtained by deleting the edges $u_{1}u_{1}'$ and $u_{1}u_{2}$, see Fig.~\ref{bi-star-fig}.

\begin{figure}[ht!]
\begin{center}
\epsfig{file=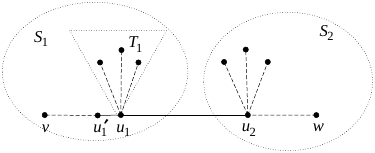, scale=1.5}
\end{center}
\vspace{-0.5cm}
\caption{The configuration of a bi-star}
\label{bi-star-fig}
\end{figure}

\begin{lemma}\label{bi-star-cross-S-set}
Let $T$ be a bi-star, and $S_{1}, S_{2}$ its corresponding star-like trees. If $v\in V(S_i)$, $i\in \{1,2\}$, then a Steiner $3$-ecc $v$-set has a vertex in $V(S_{3-i})$.
\end{lemma}

\begin{proof}
By symmetry we may assume without loss of generality that $v\in V(S_{1})$. Let $S_{v}$ be a Steiner $3$-ecc $v$-set.  By Lemma \ref{distance-growing} and the definition of bi-stars, the other two vertices in $S_{v}$ must be leaves of $T$. Suppose on the contrary that no vertex in $S_{v}\setminus \{v\}$ is in $S_{2}$. Then there must be a vertex $u\in S_{v}\setminus \{v\}$ such that $u$ is a leaf in $S_{1}$. Let $S'_{v}=S_{v}\setminus \{u\}\cup \{w\}$. Then $S'_{v}$ is a $3$-set containing $v$ and the Steiner tree on $S'_{v}$ is strictly larger than that on $S_{v}$, since there is an edge between $u_{1}$ and $u_{2}$, a contradiction. 
\end{proof}

We can define a $\tau$-transformation on a bi-star $T$, which is to transform $T$ into a new tree $T'$, with $T' = \tau(T) = T\setminus\{(u_{1},x): x\in N_{T_{1}}(u_{1})\}\cup \{(u_{2},y): y\in N_{T_{1}}(u_{1})\}$, see Fig.~\ref{tau-trans-fig}. Obviously, $T'$ is a star-like tree with exactly one branching vertex.

\begin{figure}[ht!]
\begin{center}
\epsfig{file=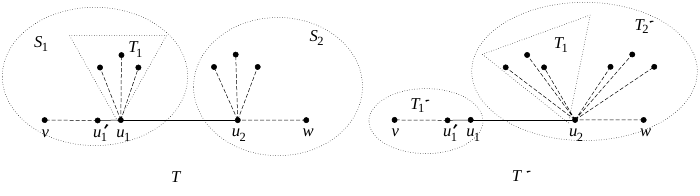, scale=1.5}
\end{center}
\vspace{-0.5cm}
\caption{The $\tau$-transformation: $T'=\tau(T)$ where $T$ is bi-star}
\label{tau-trans-fig}
\end{figure}

Deleting the edge $u_{1}u_{2}$ from $T'$ results in two subtrees. Let $T_{1}'$ be the subtree rooted at $u_{1}$, and $T_{2}'$ the subtree rooted at $u_{2}$, see Fig.~\ref{tau-trans-fig}. The proof of the following lemma is similar to that of Lemma~\ref{bi-star-cross-S-set}, hence we omit it. 

\begin{lemma}\label{tau-trans-proerty}
Let $T$ be a bi-star, let $T'=\tau(T)$, and let $T_{1}', T_{2}'$, be the the subtrees defined in the above paragraph. If $v\in V(T'_i)$, $i\in \{1,2\}$, then a Steiner $3$-ecc $v$-set has a vertex in $V(T'_{3-i})$.
\end{lemma}

\begin{theorem}\label{maintain-ecc-tau-trans}
Let $T$ be a bi-star and let $T' = \tau(T)$. Then for every vertex $v\in V(T)=V(T')$ the Steiner $3$-eccentricity of $v$ is the same in both $T$ and $T'$.
\end{theorem}

\begin{proof}
In this proof, we use the notations from Fig.~\ref{tau-trans-fig}. Let $P_{1}$ be the path between $u_{1}$ and $v$, and let $P_{2}$ be the path between $u_{2}$ and $w$.  The vertex set $V(T)$ is partitioned into the following four subsets: $V_{1} = V(P_{1})$, $V_{2} = V(P_{2})$, $V_{3}=V(S_{1})\setminus V_{1}$ and $V_{4}=V(S_{2})\setminus V_{2}$. We are going to prove that for each set $V_{i}$, $i\in\{1,2,3,4\}$, and for every vertex $v\in V_{i}$, we have $\ec_{3}(v,T)=\ec_{3}(v,T')$.

\medskip\noindent
{\bf Case 1}: $v\in V_{1}= V(P_{1})$.\\
The position of $v$ does not change after the $\tau$-transformation. Let $S_{T}(v)$ and $S_{T'}(v)$ be Steiner $3$-ecc $v$-sets in $T$ and $T'$, respectively. By Lemmas~\ref{bi-star-cross-S-set}, \ref{tau-trans-proerty} and \ref{distance-growing}, without loss of generality, let $w$ be in both $S_{T}(v)$ and $S_{T'}(v)$. Now we consider the third vertex in $S_{T}(v)$ as well as $S_{T'}(v)$. Obviously, in $S_{T'}(v)$ the third vertex must be a leaf in $V(T_{2}')\setminus\{w\}$. So no matter whether the third vertex of $S_{T}(v)$ is in $V_{3}$ or in $V_{4}$, the size of Steiner tree on $S_{T}(v)$ and that on $S_{T'}(v)$ are the same.

\medskip\noindent
{\bf Case 2}: $v\in V_{2}= V(P_{2})$.\\
The proof in this case is similar to Case 1.

\medskip\noindent
{\bf Case 3}: $v\in V_{2}=V(S_{1})\setminus V_{1}$.\\
After the $\tau$-transformation, the position of $v$ is changed. Let $S_{T}(v)$ and $S_{T'}(v)$ be Steiner $3$-ecc $v$-sets in $T$ and $T'$, respectively. By Lemmas~\ref{bi-star-cross-S-set}, \ref{tau-trans-proerty} and \ref{distance-growing}, without loss of generality, let $w$ be in $S_{T}(v)$ and $v$ be in $S_{T'}(v)$. Now consider the third vertex in $S_{T}(v)$ as well as in $S_{T'}(v)$. It is obvious that the third vertex in $S_{T}(v)$ is a leaf of $S_{1}\cup S_{2}\setminus\{w\}$, while the third vertex in $S_{T'}(v)$ is a leaf of $S_{1}\cup S_{2}\setminus\{v\}$. Therefore, the sizes of Steiner trees on $S_{T}(v)$ and on $S_{T'}(v)$ are the same.

\medskip\noindent
{\bf Case 4}: $v\in V_{2}=V(S_{2})\setminus V_{2}$.\\
The proof in this case is also similar to Case 1.
\end{proof}

\section{Some applications of the $\pi$-transformation}
\label{applic}

As an application of the $\pi$-transformation, we establish in this section several lower and upper bounds on the average Steiner $3$-eccentricity of trees in terms of the order, the maximum degree, the number of pendent vertices, the matching number, the independent number, the diameter, and the radius.

\subsection{Lower and upper bounds on general trees}

\begin{theorem}\label{bounds-on-general-trees}
If $T$ is a tree on $n$ vertices, then
$$3-\frac{1}{n}\leq \aec_{3}(T)\leq n-1\,.$$
The lower bound is achieved when $T\cong K_{1,n-1}$, while  the upper bound is achieved when $T\cong P_n$.
\end{theorem}

\begin{proof}
Let $T$ be an arbitrary tree of order $n$.  Then repeatedly apply the $\pi$-transformation on $T$ until no further $\pi$-transformation is possible.

We claim that in the last step we must necessarily arrive at $K_{1,n-1}$. See Fig.~\ref{fig:to-star} for an example of such a procedure. Suppose on the contrary that in the last step we do not arrive at $K_{1,n-1}$. Let $T'$ be the tree achieved in the last step. In this case, there are at least two internal vertices. Without loss of generality, let $u_{1}$ and $u_{2}$ be two distinct internal vertices in $T'$. Let $u_{3}$ be the unique neighbor of $u_{1}$ in the unique path between $u_{1}$ and $u_{2}$ in $T'$. Let $T'_{1}$ and $T'_{2}$ be the two subtrees respectively rooted at $u_{1}$ and $u_{3}$ in the forest which is obtained by deleting the edge $u_{1}u_{2}$ from $T'$ (see Fig.~\ref{star}).  Now one can do a $\pi$-transformation to obtain a new tree $T''$ defined as $T''=T\setminus\{(u_{1},x):x\in N_{T'_{1}}(u_{1})\}\cup\{(u_{3},y):y\in N_{T'_{1}}(u_{1})\}$. This contradicts the fact that $T'$ is achieved in the last step. Hence, in the last step we must arrive at $K_{1,n-1}$.

 \begin{figure}
  \centering
  \subfigure[]{
    \label{star} 
     \epsfig{file=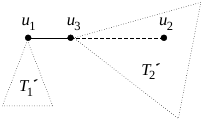,scale=1.5}
     }
  \hspace{1in}
  \subfigure[]{
    \label{path} 
    \epsfig{file=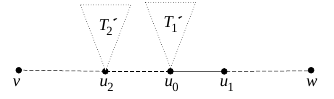, scale = 1.2}
    }
  \caption{To prove the bounds on general graphs}
  \label{bound-general} 
\end{figure}

By Theorem~\ref{alpha-trans-not-decrease}, during the process, each $\pi$-transformation does not increase the average Steiner $3$-eccentricity. Hence $\aec_{3}(K_{1,n-1}) \le \aec_3(T)$.

\begin{figure}[ht!]
\begin{center}
\epsfig{file=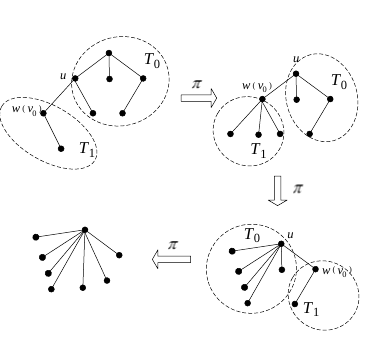, scale=1.5}
\end{center}
\vspace{-0.5cm}
\caption{Transforming a tree with a sequences of $\pi$-transformations to a star}
\label{fig:to-star}
\end{figure}

On the other hand, repeatedly apply the $\pi^{-1}$-transformation on $T$ as long as it is possible. We claim that in the last step we must necessarily arrive at the path $P_{n}$, see Fig.~\ref{fig:to-path} for an example. Suppose on the contrary that in the last step we do not arrive at $P_{n}$. Let $T'$ be the tree obtained in the last step. Then $T'$ contains at least one vertex of degree at least $3$, let $u_{0}$ be such a vertex. Let $u_{1}$ be the unique neighbor of $u_{0}$ on a longest path starting at $u_{0}$ in $T'$, and let $w$ be the other endpoint of the longest path, see Fig.~\ref{path}. Let $v$ be the leave of $T'$ such that the path between $v$ and $u_{0}$ has the smallest length among all paths starting at $u_{0}$ to all leaves. If the path between $v$ and $u_{0}$ is a pendent path of $T'$, then let $u_{2}$ be the neighbor of $u_{0}$ on this path. Otherwise, let $u_{2}$ be a branching vertex of $T'$ such that $u_{2}$ is on the path between $v$ and $u_{0}$ and the path between $v$ and $u_{2}$ is a pendent path of $T'$, see Fig. \ref{path}.   Hence the length of the pendent path between $v$ and $u_{2}$ is strictly less than that between $u_{0}$ and $w$. This implies that one more $\pi^{-1}$-transformation can be done on $T'$.  This contradicts the fact that $T'$ is the tree obtained in the last step. So we finally must arrive at the path $P_{n}$.

\begin{figure}[ht!]
\begin{center}
\epsfig{file=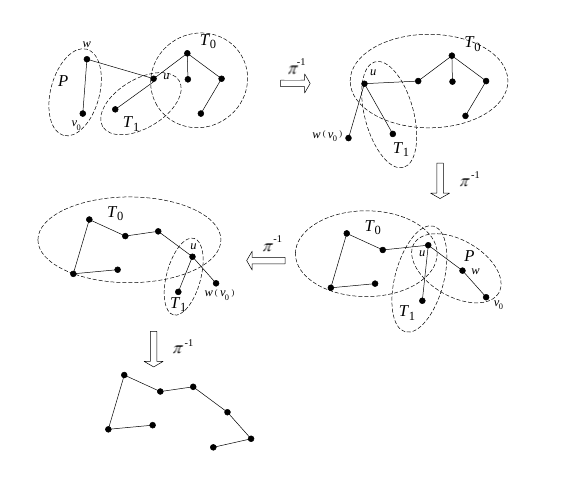, scale=1.3}
\end{center}
\vspace{-1.0cm}
\caption{Transforming a tree with a sequences of $\pi^{-1}$-transformations to a path}
\label{fig:to-path}
\end{figure}

 By Corollary~\ref{inverse-alpha-trans}, at each step of this process  the average Steiner $3$-eccentricity does not decrease, hence $\aec_3(P_{n})\ge \aec_3(T)$. Using the values from Proposition~\ref{aecc-on-special-graphs} we thus have $3 - \frac{1}{n} = \aec(S_{1,n-1}) \le \aec_3(T) \le \aec_3(P_n) = n - 1$ and we are done.
\end{proof}

\subsection{An upper bound on trees with maximum degree}

A {\em broom} $B(n, \Delta)$ is a tree obtained from $K_{1,\Delta}$ by attaching a path of length $n-\Delta-1$ to an arbitrary pendent vertex of the star. See Fig.~\ref{broom} for an example.

\begin{figure}[ht!]
\begin{center}
\epsfig{file=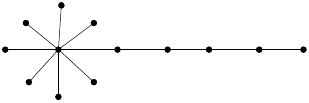, scale=1.3}
\end{center}
\vspace{-0.5cm}
\caption{The broom $B(13,8)$}
\label{broom}
\end{figure}

\begin{theorem}\label{max-deg-tree-upper-bound}
If $T$ is a tree of order $n = n(T)$ and maximum degree $\Delta = \Delta(T)$, then
$$\aec_{3}(T)\leq \aec_{3}(B(n,\Delta))=n-\Delta + 1 + \frac{\Delta}{n}\,.$$
\end{theorem}

\begin{proof}
Let $T$ be a tree with $n = n(T)$ and $\Delta = \Delta(T)$, and let $r$ be the vertex of $T$ with degree $\Delta$. Consider $T$ as a tree rooted in $r$. Let $T_{1},\ldots, T_{\Delta}$ be the maximal subtrees of $T$ each of which do not contain $r$, but exactly one of the neighbor is $r$. We may also consider these $\Delta$ trees to be rooted at $r$. Repeatedly apply the $\pi^{-1}$-transformation on each subtree $T_{i}$, until no more such transformations could be done. By the proof of Theorem \ref{bounds-on-general-trees}, in the last step, every $T_{i}$ becomes a path. Let $T'$ be the final tree in this procedure.

When all subtrees turn into paths, we can further proceed the $\pi^{-1}$-transformation on $T'$ until we arrive at the broom $B(n,\Delta)$, see Fig.~\ref{fig:to-broom} for an example. Now we will elaborate the process.

\begin{figure}[ht!]
\begin{center}
\epsfig{file=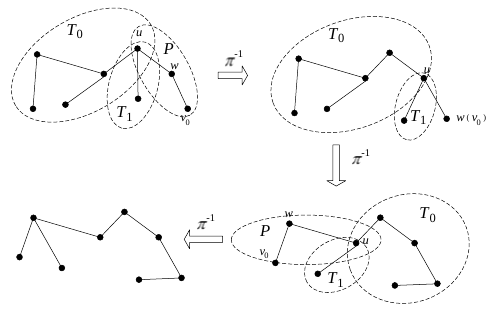, scale=1.5}
\end{center}
\vspace{-0.5cm}
\caption{Transforming a tree with a sequences of $\pi^{-1}$-transformations to a broom}
\label{fig:to-broom}
\end{figure}

One can imagine that $T'$ is a tree in which there are $\Delta$ paths attached to a star $K_{1,\Delta}$. The star has the central vertex $r$. If there is no more than one neighbor of $r$ each of which has degree $2$, then $T'$ is just a broom. In this case, the theorem holds trivially. Now we consider that there are at least two neighbors in $N_{T'}(r)$ each of which has degree $2$. Without loss of generality, let $r_{1}, r_{2}\in N_{T'}(r)$ be two distinct neighbors of $r$, see Fig.~\ref{arrive-broom}. In ťhe figure, $u$ and $v$ are both leaves of $T'$.  If the size of the path between $r_{1}$ and $u$ is less than that between $r$ and $v$, then apply a $\pi^{-1}$-transformation on $T'$ to obtain a new tree $T''$ with  $T'' = T'\setminus\{(r,x):x\in N_{T_{1}'}(r)\}\cup \{(r_{1},y):y\in N_{T_{1}'}(r)\}$. Otherwise, also apply a $\pi^{-1}$-transformation on $T'$ to obtain another new tree $T'''$ with $T''' = T'\setminus\{(r,x):x\in N_{T_{1}'}(r)\}\cup \{(r_{2},y):y\in N_{T_{1}'}(r)\}$. Whenever there are two such different neighbors of $r$, whose degrees are both $2$, we employ the above procedure on the tree.

\begin{figure}[ht!]
\begin{center}
\epsfig{file=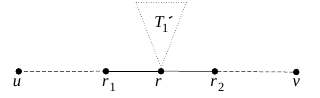, scale=1.5}
\end{center}
\vspace{-0.5cm}
\caption{To achieve a broom}
\label{arrive-broom}
\end{figure}

We claim that in the last step we will achieve a broom with maximum degree $\Delta$. Obviously, the degree of the center vertex in the star $K_{1,\Delta}$ remains to be $\Delta$ all the way in the whole procedure. So the maximum degree of the tree obtained in the last step is $\Delta$. Now we show that in the final step we must arrive at a broom. Suppose on the contrary that we do not achieve a broom in the last step. Let $T_{1}$ be the tree in the last step and $r$ be the central vertex of the star $K_{1,\Delta}$. In this case, there must be two neighbors of $r$ each of which has degree $2$. This means that one more $\pi^{-1}$-transformation can be done on $T_{1}$, just like the procedure in the above paragraph. This contradicts the fact that the tree is achieved in the last step. So finally we must achieve a broom.

By Corollary~\ref{inverse-alpha-trans}, during the whole process the Steiner $3$-eccentricity does not decrease. This implies that $\aec_{3}(T)\leq \aec_{3}(B(n,\Delta))$. Finally, the broom $B = B(n,\Delta)$ has $\Delta$ leaves, and $\ec_3(v,B) = n - \Delta + 2$ holds for each of its leaves $v$. For each of the other $n-\Delta$ vertices $w$ of $B$ we have $\ec_3(w,B) = n - \Delta + 1$. Hence $\aec_3(B) = (\Delta(n-\Delta+2) + (n-\Delta)(n-\Delta+1))/n = n + 1 - \Delta + \frac{\Delta}{n}$, and we are done.
\end{proof}

\subsection{A lower bound on trees with constant number of leaves}

A {\em starlike tree} is a tree with exactly one vertex of degree at least three. In other words, a starlike tree is a tree obtained by attaching to an isolated vertex $t\ge 3$ pendant paths. If the lengths of these pendant paths pairwise differ by at most one, then the starlike tree is called \emph{balanced}. Note that if $T$ is a balanced starlike tree of order $n$ and with $p$ leaves, then it is uniquely determined (up to isomorphism); we will denote it by $BS_{n,p}$.

\begin{theorem}\label{lower-bound-on-given-num-leaves}
Let $T$ be a tree of order $n\ge 2$ and with $p$ pendent vertices. Then
$$\aec_{3}(T)\geq \aec_{3}(BS_{n,p})\,.$$
\end{theorem}

\begin{proof}
 If $T$ has more than two branching vertices, by Lemma \ref{pendent-trans-2}, one can do a pendent $\pi$-transformation on $T$. Recall that the $\pi$-transformation does not increase the average Steiner $3$-eccentricity.

 Now suppose $T$ has exact two branching vertices. If the distance of these two branching vertices is at least two, then by Lemma \ref{pendent-trans-1}, one can do a pendent $\pi$-transformation on $T$. If the distance of these two branching vertices is exact one, and there are two pendent paths such that they have different length and attach to distinct branching vertices, then by Lemma \ref{pendent-trans-3}, one can also do a pendent $\pi$-transformation on $T$. All the processes does not increase the average Steiner $3$-eccentricity. If the distance of these two branching vertices is exact one, but every pendent path has the same length, then by Theorem \ref{maintain-ecc-tau-trans}, one can do a $\tau$-transformation to obtain a new tree with only one branching vertices.

 Now we consider the case that there is an unique branching vertex in the tree $T$. If there are two distinct pendent paths such that their lengths differ by at most one, then $T$ is trivially a balanced star-like tree. Otherwise,  by Lemma \ref{pendent-trans-4}, one can also do a $\pi$-transformation on $T$.

 To sum all, all the transformations do not increase the average Steiner $3$-eccentricity. Finally we arrive at the starlike tree $BS_{n,p}$, see Fig.~\ref{fig:to-balance} for an example.
\begin{figure}
\begin{center}
\epsfig{file=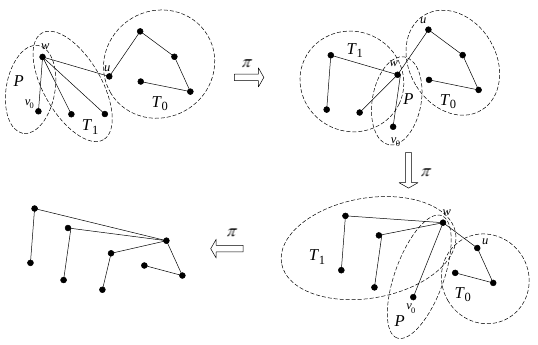, scale=1.5}
\end{center}
\vspace{-0.5cm}
\caption{Transforming a tree with a sequences of $\pi$-transformations to a balanced starlike tree}
\label{fig:to-balance}
\end{figure}

Since in all the transformations made to reach $BS_{n,p}$ the average Steiner $3$-eccentricity has not increased, we conclude that $\aec_{3}(BS_{n,p}) \le \aec_{3}(T)$.
\end{proof}

\subsection{Lower bounds on trees with matching and independence number}

If $m\ge 3$ and $m+2\le n\le 2m+1$, then let $T_{n,m}$ be a tree obtained from $K_{1,m}$ by respectively adding a pendent edge to its $n-m-1$ pendent vertices. Hence we add a leaf to at least one and not more than $m$ leafs of $K_{1,m}$. Note also that $n(T_{n,m}) = n$. Observe further that $\alpha(T_{n,m}) = m$ and $\beta(T_{n,m}) = n-m$, where $\alpha(G)$ and $\beta(G)$ are the independence and the matching number of $G$, respectively.

\begin{theorem}\label{lower-bound-on-given-max-match}
If $T$ is a tree with $n = n(T)$ and $\beta = \beta(T) \geq 2$, then
$$\aec_{3}(T)\geq \aec_{3}(T_{n,n-\beta})\,.$$
\end{theorem}
\begin{proof}
Let $M$ be a maximum matching of $T$, so that $|M| = \beta$. Set further $\ell = \ell(T)$. If $e\in M$, then at most one of the endpoints of $e$ is a leaf, hence $\ell  \leq \beta+(n-2\beta)=n-\beta$. By  Theorem~\ref{lower-bound-on-given-num-leaves}, we have $\aec_{3}(T)\geq \aec_{3}(BS_{n,\ell})$. Applying Theorem~\ref{alpha-trans-not-decrease} again we  can then estimate that
$$\aec_{3}(T)\geq \aec_{3}(BS_{n,\ell}) \geq \aec_{3}(BS_{n,n-\beta}) = \aec_{3}(T_{n,n-\beta})\,,$$
and we are done.
\end{proof}

For trees with perfect matchings, Theorem~\ref{lower-bound-on-given-max-match} together with a straightforward computation of $\aec_{3}(T_{n,n/2})$ yields the following consequence.

\begin{corollary}\label{lower-bound-on-given-perfect-match}
If $T$ be a tree of order $n$ with a perfect matching, then
     $$\aec_{3}(T)\geq \aec_{3}(T_{n,n/2})=\begin{cases}
     3;                                   &  n=4,\\
     \frac{9}{2};                         &  n=6,\\
     \frac{11}{2}-\frac{2}{n};            &  n\geq 8.
     \end{cases}$$
\end{corollary}

We next give a bound with the independence number of a tree.

\begin{theorem}\label{lower-bound-on-given-max-independent-num}
If $T$ be a tree of order $n$ and $\alpha = \alpha(T)$, then
     $$\aec_{3}(T)\geq \aec_{3}(T_{n,\alpha})\,.$$
\end{theorem}

\begin{proof}
Set again $\ell = \ell(T)$. Clearly, $\alpha \ge \ell(T)$.  By Theorem~\ref{lower-bound-on-given-num-leaves}, we have $\aec_{3}(T)\geq \aec_{3}(BS_{n,\ell})$. By the aid of Theorem~\ref{alpha-trans-not-decrease} we conclude that $\aec_{3}(T) \geq \aec_{3}(BS_{n,\ell}) \geq \aec_{3}(BS_{n,\alpha})$.
\end{proof}

\subsection{Lower bounds on trees with constant diameter or radius}

Recall that the {\em diameter} ${\rm diam}(G)$ and the {\em radius} ${\rm rad}(G)$ of a graph $G$ are the maximum and the minimum, respectively, eccentricity of the vertices of $G$.  The {\em center} of $G$ is the set of its vertices with minimum eccentricity. Recall also that the center of a tree consists either of a single vertex or of two adjacent vertices.

Let $T_{n,d}(p_{1},\ldots, p_{d-1})$ be a tree of order $n$ obtained from a path $P_{d+1}=v_{0}v_{1}\ldots v_{d}$ by attaching $p_{i}\geq 0$ pendent vertices to $v_{i}$ for every $i\in [d-1]$. Clearly, as the order of $T'_{n,d}(p_{1},\ldots, p_{d-1})$ is $n$, we must have $\Sigma_{i=1}^{d-1}p_{i}=n-d-1$. In the special case when $d$ is even and all the $n-d-1$ vertices are attached to the vertex $v_{d/2}$, we briefly denote the tree with $T'_{n,d}$.  Similarly, if $d$ is odd, then let $T'_{n,d}$ denote the graph in which $\lfloor (n-d-1)/2\rfloor$ vertices are attached to $v_{\lceil d/2\rceil}$ and $\lceil (n-d-1)/2\rceil$ vertices are attached to $v_{\lceil d/2\rceil}$.

\begin{theorem}\label{lower-bound-on-given-daameter}
If $T$ is a tree of order $n$ and ${\rm diam}(T) = d$, then
$$\aec_{3}(T)\geq \aec_{3}(T'_{n,d})\,.$$
\end{theorem}

\begin{proof}
Let $T$ be a tree as stated and let $P = v_{0}v_{1}\ldots v_{d}$ be a longest path in $T$. Since $P$ is a longest path in $T$, both $v_{0}$ and $v_{d}$ are leaves of $T$. For $i\in [d-1]$ let $T_i$ be a maximal subtree of $T$ that contains $v_i$ but no other vertex of $P$. Consider $T_i$ as a rooted tree with the root $v_i$. Then the depth of the rooted tree $T_{i}$ is at most the minimum of the lengths of the $v_0,v_i$-subpath of $P$ and the $v_i,v_d$-subpath of $P$, that is, at most $\min\{i, d-i\}$.  Therefore, for each $i\in [d-1]$, we can repeatedly apply the $\pi$-transformation on the subtree $T_{i}$ respected to $T$ so than $T_{i}$ turns into a star rooted at $v_{i}$. The average Steiner $3$-eccentricity has not increased along the way. After this procedure is over, $T_{n,d}(p_{1},\ldots, p_{d-1})$ is constructed. Afterwards we repeatedly apply the $\pi$-transformation on each pendent vertex attached to $v_i$ for each $i\in [d-1]$, to arrive at $T'_{n,d}$.
\end{proof}

Note that if $d$ is odd, then we could define $T'_{n,d}$ also by arbitrary distributing the $n-d-1$ vertices that are attached to $v_{\lceil d/2\rceil}$ and to $v_{\lceil d/2\rceil}$. That is, any such tree can serve for the lower bound of Theorem~\ref{lower-bound-on-given-daameter}.

If the center of a tree $T$ contains only one vertex, then ${\rm diam}(T) = 2\,{\rm rad}(T)$, and if the center of $T$ consists of two vertices, ${\rm diam}(T) = 2\,{\rm rad}(T)-1$. Hence Theorem~\ref{lower-bound-on-given-daameter} yields the following consequence.

\begin{corollary}\label{lower-bound-on-given-reaius}
If $T$ is a tree of order $n$ and $r = {\rm rad}(G)$, then $\aec_{3}(T)\geq \aec_{3}(T'_{n,2r-1})$.
\end{corollary}

\section{Concluding remarks}
\label{conclusions}

We have derived several lower and upper bounds for the average Steiner $3$-eccentricity on a tree with different constrained parameters. It would be interesting to see if and how these bounds extend to $k\geq 4$.

Just a little research has been done by now on the (average) Steiner $k$-eccentricity for $k\geq 3$. Hence a lot of work still has to be done.


\frenchspacing
\begin{thebibliography}{99}

\bibitem{Ilic2020OptimalAF}
  A.~Ili\'{c},
  Optimal algorithm for computing Steiner 3-eccentricities of trees,
  or arXiv:2008.09299v2 [cs.DS] ( 21 Feb 2021). 

\bibitem{Anand2012}
  B.~S.~Anand, M.~Changat, S. Klav\v{z}ar, I.~Peterin,
  Convex sets in lexicographic product of graphs,
  Graphs Combin.\ 28 (2012) 77--84.

\bibitem{Beineke1996}
  L.~W.~Beineke, O.~R.~Oellermann, R.~E.~Pippert,
  On the Steiner median of a tree,
  Discrete Appl.\ Math.\ 68 (1996) 249--258.

\bibitem{casablanca-2019}
  R.~M.~Casablanca, P.~Dankelmann,
  Distance and eccentric sequences to bound the {W}iener index, {H}osoya polynomial and the average eccentricity in the strong products of graphs,
  Discrete Appl.\ Math.\ 263 (2019) 105--117.

\bibitem{Chartrand1989}
  G.~Chartrand, O.~R.~Oellermann, S.~Tian, H.~B.~Zou,
  Steiner distance in graphs,
  \v{C}asopis P\v{e}st.\ Mat.\ 114 (1989) 399--410.

\bibitem{Dankelmann1997}
  P.~Dankelmann, H.~C.~Swart, O.~R.~Oellermann,
  The average Steiner distance of graphs with prescribed properties,
  Discrete Appl.\ Math.\ 79 (1997) 91--103.

\bibitem{Dankelmann1996}
  P.~Dankelmann, O.~R.~Oellermann, H.~C.~Swart,
  The average Steiner distance of a graph,
  J.\ Graph Theory 22 (1996) 15--22.

\bibitem{dankelmann2004}
  P.~Dankelmann, W.~Goddard, H.~C.~Swart,
  The average eccentricity of a graph and its subgraphs,
  Util.\ Math.\ 65 (2004) 41--51.

\bibitem{dankelmann2014}
  P.~ Dankelmann, S.~Mukwembi,
  Upper bounds on the average eccentricity,
  Discrete Appl.\ Math.\ 167 (2014) 72--79.

\bibitem{dankelmann2019}
  P.~ Dankelmann, F.~J.~Osaye, S.~Mukwembi, B.~G.~Rodrigues,
  Upper bounds on the average eccentricity of {$K_3$}-free and {$C_4$}-free graphs,
  Discrete Appl.\ Math.\ 270 (2019) 106--114.

\bibitem{Garey1979}
  M.~R.~Garey, D.~S.~Johnson,
  Computers and Intractability: A Guide to the Theory of NP-Completeness,
  Freeman \& Company, New York, 1979.

\bibitem{ghorbani}
  M.~Ghorbani, X.~Li, H.~R.~Maimani, Y.~Mao, S.~Rahmani, M~R.~Parsa,
  Steiner (revised) Szeged index of graphs,
  MATCH Commun.\ Math.\ Comput.\ Chem.\ 82 (2019) 733--742.

\bibitem{Gologranc}
  T.~Gologranc,
  Steiner convex sets and Cartesian product,
  Bull.\ Malays.\ Math.\ Sci.\ Soc.\ 41 (2018) 627--636.

\bibitem{GutmanInPress}
  I.~Gutman,  X.~ Li, Y.~ Mao,
  Inverse problem on the Steiner Wiener index,
  Discuss.\ Math.\ Graph Theory 38 (2017) 83--95.

\bibitem{Gutman2016}
  I.~Gutman, On Steiner degree distance of trees,
  Appl.\ Math.\ Comput.\ 283 (2016) 163--167.

\bibitem{Gutman2015}
  I.~Gutman, B.~Furtula, X.~Li,
  Multicenter Wiener indices and their applications,
  J.\ Serb.\ Chem.\ Soc.\ 80 (2015) 1009--1017.

\bibitem{Hwang1992}
  F.~K.~Hwang, D.~S.~Richards, P.~Winter,
  The Steiner Tree Problem,
  North-Holland, Amsterdam, 1992.

\bibitem{Li2016Wiener}
  X.~Li, Y.~Mao, I.~Gutman,
  The Steiner Wiener index of a graph,
  Discuss.\ Math.\ Graph Theory 36 (2016) 455--465.

\bibitem{Li2018Wiener}
  X.~Li, Y.~Mao, I.~Gutman,
  Inverse problem on the Steiner Wiener index,
  Discuss.\ Math.\ Graph Theory 38 (2018) 83--95.

\bibitem{Li2020TheS}
  X.~Li, G.~Yu, S.~Klavv\v{z}ar, J.~Hu and B.~Li,
  The Steiner $k$-eccentricity on trees,
  arXiv:2008.07763 [math.CO]  (18 Aug 2020).
  
\bibitem{lu}
  L.~Lu, Q.~Huang, J.~Hou, X.~Chen,
  A sharp lower bound on the Steiner Wiener index for trees with given diameter,
  Discrete Math.\ 341 (2018) 723--731.

\bibitem{Mao2017Survey}
  Y.~Mao, Steiner distance in graphs---a survey,
  arXiv:1708.05779 [math.CO]  (18 Aug 2017).

\bibitem{mao-2017}
  Y.~Mao,
  Steiner Harary index,
  Kragujevac J.\ Math.\ 42 (2018) 29--39.

\bibitem{mao-2018}
  Y.~Mao, E.~Cheng, Z.~Wang,
  Steiner distance in product networks,
   Discrete Math.\ Theor.\ Comput.\ Sci.\ 20 (2018) paper 8, 25 pp.

\bibitem{mao-2018b}
  Y.~Mao, K.~C.~Das,
  Steiner Gutman index,
  MATCH Commun.\ Math.\ Comput.\ Chem.\ 79 (2018) 779--794.

\bibitem{pei2019}
   L.-D.~Pei, X.-F.~Pan, J.~Tian G.-Q.~Peng,
  On {A}uto{G}raphi{X} conjecture regarding domination number and average eccentricity,
  Filomat 33 (2019) 699--710.

\bibitem{Tratnik2019}
  N.~Tratnik, On the Steiner hyper-Wiener index of a graph,
  Appl.\ Math.\ Comput.\ 337 (2019) 360--371.

\bibitem{Wang2017}
  Z.~Wang, Y.~Mao, C.~Melekian, E.~Cheng,
  Steiner Distance in join, corona, cluster, and threshold graphs,
  J.\ Inform.\ Sci.\ Eng.\ 35 (2019) 721--735.

\bibitem{Wang2018}
  C.~Wang, C.~Ye, S.~Zhang, Z.~Wang, H.~Li,
  The Steiner Wiener index of trees with given diameter,
  Util.\ Math.\ 108 (2018) 45--52.

\bibitem{zhang-2019}
  J.~Zhang, G.-J.~Zhang, H.~Wang, X.~D.~Zhang,
  Extremal trees with respect to the {S}teiner {W}iener index,
  Discrete Math.\ Algorithms Appl.\ 11 (2019) paper 1950067, 16 pp.

\end{thebibliography}
\end{document}